\documentclass[11pt]{amsart}

\usepackage{amsmath,amssymb,amsthm} 
\usepackage[small,nohug,heads=littlevee]{diagrams} \diagramstyle[labelstyle=\scriptstyle] 
\usepackage{eucal}

\newcommand{\bbar}[1]{{\overline{#1}}}

\newcommand{\F}{\mathbb F} 
\newcommand{\Gal}{{\mathrm{Gal}\, }} 
\newcommand{\GL}{{\mathbf{GL} }} 
 
\newcommand{\sepa}{{\mathrm{sep}\, }} 
\newcommand{\Map}{{\mathrm{Map}\, }}

\newcommand{\A}{\mc A} 
 
\newcommand{\bb}[1]{{\boldsymbol{#1}}} 
\renewcommand{\b}[1]{{\mathbf{#1}}} 
 
\newcommand{\G}{{\mathbf{G} }} 
\newcommand{\Sol}{\mathrm{Sol}} 
\newcommand{\im}{\mathrm{im}} 
\newcommand{\mc}[1]{{\mathcal{#1}}} 

\newtheorem{theorem}{Theorem}[section] 
\newtheorem{lemma}[theorem]{Lemma} 
\newtheorem{corollary}[theorem]{Corollary} 
\newtheorem{proposition}[theorem]{Proposition}

\theoremstyle{definition} 
\newtheorem{definition}{Definition} 
\newtheorem{remark}{Remark} 
\newtheorem{example}{Example}

\begin{document}

\title[Generic extensions and generic Polynomials]{Generic extensions and generic Polynomials for multiplicative groups}

\author{Jorge Morales} \address{Department of Mathematics, Louisiana State University, Baton Rouge, LA 70803, USA.} \email{morales@math.lsu.edu}

\author{Anthony Sanchez$^*$} \address{Department of Mathematics, Arizona State University, Tempe, AZ 85281, USA.} \email{anthony.sanchez.1@asu.edu} \thanks{$^*$Research conducted at the 2012 Louisiana State University Research Experience for Undergraduates (REU) site supported by the National Science Foundation REU Grant DMS -0648064}
\begin{abstract}
	Let $\A$ be a finite-dimensional algebra over a finite field $\F_q$ and let $G=\A^\times$ be the multiplicative group of $\A$. In this paper, we construct explicitly a generic Galois $G$-extension $S/R$, where $R$ is a localized polynomial ring over $\F_q$, and an explicit generic polynomial for $G$ in $\dim_{\F_q}(\A)$ parameters. 
\end{abstract}

\maketitle

\section{Introduction}

An important and classical problem in Galois theory is to describe for a field $k$ and a finite group $G$ all Galois extensions $M/L$ with Galois group $G$, where $L$ is a field containing $k$. This can be done by means of a {\em generic polynomial}, that is a polynomial $f(Y;t_1,\ldots,t_m)$ with coefficients in the function field $k(t_1,\ldots,t_m)$ and Galois group $G$ such that every Galois $G$-extension $M/L$, with $L\supset k$, is the splitting field of $f(Y;\xi_1,\ldots,\xi_m)$ for a suitable $(\xi_1,\ldots,\xi_m)\in L^m$.

A related construction is that of {\em generic extensions} introduced by Saltman \cite{Saltman:1982uq}. These are Galois $G$-extensions of commutative rings $S/R$, where $R=k[t_1,\dots,t_m,1/d]$ and $d$ is a nonzero polynomial in $k[t_1,\dots,t_m]$, such that every Galois $G$-algebra $M/L$, where $L$ is a field containing $k$, is of the form $M\simeq S\otimes_\varphi L$ for a suitable homomorphism of $k$-algebras $\varphi:R\to L$.

Over an infinite ground field $k$, the existence of generic polynomials is equivalent to the existence of generic extensions as shown by Ledet \cite{Ledet:2000kx}, but the dictionary, at least in the direction $\{\mathrm{polynomials}\}\to \{\mathrm{extensions}\}$, is not straightforward. 

In this paper, we construct explicitly both a generic extension and a generic polynomial for groups of the form $G=\A^\times$, where $\A$ is a finite-dimensional $\F_q$-algebra and $k$ is an infinite field containing $\F_q$. Both constructions are based in the theory of Frobenius modules as developed by Matzat \cite{Matzat:2004kx}. An important ingredient is Matzat's ``lower bound'' theorem as formulated in \cite[Theorem 3.4]{Albert:2011vn} that we use to show that the extensions (respectively, polynomials) we construct have the required Galois group.

The number of parameters in our construction is not optimal. For example, if $\A=M_n(\F_q)$, then our method produces a polynomial in $n^2$ parameters, as opposed to the standard generic polynomial for $\GL_n(\F_q)$ that needs only $n$ parameters \cite{Abhyankar:2000ys}, \cite[Section 1.1]{Jensen:2002bh}. However, our method has the advantage of being general for all groups of the form $\A^\times$, where $\A$ is any finite-dimensional algebra over $\F_q$.

We are indebted to the referee for her/his pertinent and useful comments.

\tableofcontents

\section{Frobenius Modules} In this section we recall the basic theory and definitions relating to Frobenius modules for convenience of the reader. Most of the material in subsections \ref{SS:Pre} -- \ref{SS:galgroup} can be found in \cite[Part I]{Matzat:2004kx}, \cite{Albert:2011vn}. We include it here for the convenience of the reader. \\

\subsection{Preliminaries}\label{SS:Pre}

Let $K$ be a field containing the finite field $\F_q$ and let $\bbar{K}$ denote an algebraic closure of $K$. 
\begin{definition}
	\label{D:fmodule} A {\em Frobenius module} over $K$ is a pair $(M,\varphi)$ consisting of a finite-dimensional vector space $M$ over $K$ and a $\F_q$-linear map $\varphi:M\to M$ satisfying 
	\begin{enumerate}
		\item $\varphi(a x)=a^q \varphi(x)$ for $a\in K$ and $x\in M$. 
		\item The natural extension of $\varphi$ to $M\otimes_K \bbar{K}\to M\otimes_K \bbar{K}$ is injective\footnote{ Note that if $K$ is not perfect, the injectivity of $\varphi:M\to M$ does not imply condition (2) above. For example, if $a\in K\setminus K^q$, the map $\varphi:K^2\to K^2$ given by $\varphi(x,y)=(x^q-a y^q,0)$ is injective over $K$ but not over $\bbar K$. }. 
	\end{enumerate}
\end{definition}
\medskip

The {\em solution space} $\Sol^\varphi(M)$ of $(M,\varphi)$ is the set of fixed points of $\varphi$, i.e.
\[ \Sol^\varphi(M)=\{x\in M\ |\ \varphi(x)=x\}, \]
which is clearly a $\F_q$-subspace of $M$. 

Let $e_1,e_2,\ldots,e_n$ be a $K$-basis of $M$. Clearly $\varphi$ is completely determined by its values on this basis. Write

\[ \varphi(e_j)=\sum_{i=1}^n a_{ij} e_i, \]
where $a_{ij}\in K$ and let $A=(a_{ij})\in M_n(K)$. Identifying $M$ with $K^n$ via the choice of this basis, we have
\[ \varphi(X)=AX^{(q)}, \]
where $X=(x_1,\ldots,x_n)^T$ and $X^{(q)}=(x_1^q,\ldots,x_n^q)^T$. Condition (2) of Definition \ref{D:fmodule} ensures that $A$ is nonsingular. We shall denote by $(K^n,\varphi_A)$ the Frobenius module determined by a matrix $A\in \GL_n(K)$.

With the above notation, the solution space $\Sol^\varphi(M)$ is identified with the set of solutions in $K$ of the system of polynomial equations 
\begin{equation}
	\label{E:system} A X^{(q)} = X. 
\end{equation}

By the Lang-Steinberg theorem (Theorem \ref{T:lang}), there is a matrix $U=(u_{ij})\in \GL_n(\bbar{K})$ such that 
\begin{equation}
	\label{E:Lang} A=U (U^{(q)})^{-1}, 
\end{equation}
where $U^{(q)}=(u_{ij}^q)$. Thus, the change of variables $Y=U^{-1} X$ over $\bbar{K}$ yields the ``trivial'' system 
\begin{equation}
	\label{E:system2} Y^{(q )}=Y, 
\end{equation}
whose solutions are exactly the vectors in $\F_q^n\subset \bbar{K}^n$. We have proved: 
\begin{proposition}
	\label{P:dim} The columns of $U$ form a basis of $\Sol^\varphi(M\otimes_K \bbar{K})$ over $\F_q$. In particular
	\[ \dim_{\F_q}\Sol^\varphi(M\otimes_K \bbar{K})=n. \]
\end{proposition}
\subsection{Separability}

We shall now show that the solutions of \eqref{E:system} are in $K_{\sepa}^n$. See \cite[Theorem 1.1c]{Matzat:2004kx} for a different argument. 
\begin{proposition}
	\label{P:separability} Let $A\in \GL_n(K)$ and let ${\b x}_1,{\b x}_2,\ldots,{\b x}_n$ be indeterminates. Then the $K$-algebra
	\[ {\mc{F}}= K[{\b X}]/\langle A {\b X}^{(q)}-{\b X}\rangle, \]
	where ${\b X}=[{\b x}_1,{\b x}_2,\ldots,{\b x}_n]^T$ and $\langle A {\b X}^{(q)}-{\b X}\rangle$ is the ideal generated by the coordinates of $A {\b X}^{(q)}-{\b X}$, is {\'e}tale over $K$. 
\end{proposition}
\begin{proof}
	Consider the change of variables ${\b Y}= U {\b X}$ over ${\bbar K}$, where $U$ is as in \eqref{E:Lang}. Then
	\[ 
	\begin{aligned}
		{\mc{F}}\otimes_K {\bbar K}&={\bbar K}[{\b Y}]/\langle {\b Y}^{(q)}-\b Y\rangle\\
		&\simeq \prod_{\F_q^n} {\bbar K}. 
	\end{aligned}
	\]
\end{proof}
\begin{corollary}
	\label{C:separability} The solutions of the system of polynomial equations $A X^{(q)} = X$ in $\bbar{K}^n$ lie in $K_\sepa^n$. In particular, the matrix $U$ of \eqref{E:Lang} is in $ \GL_n(K_\sepa)$. 
\end{corollary}
\begin{proof}
	The solutions of $A X^{(q)} = X$ are exactly the images of $\b X$ under $K$-algebra homomorphisms $\mc{F}\to \bbar K$. Since $\mc{F}/K$ is {\'etale}, so are all its quotients. This implies that the images of such homomorphisms are contained in $K_\sepa$. 
\end{proof}
\begin{definition}
	The {\em splitting field} $E$ of $(M,\varphi)$ is the subfield of $K_\sepa$ generated over $K$ by all the solutions of $AX^{(q)}=X$. 
\end{definition}
\begin{remark}
	The above definition does not depend on the choice of a basis of $M$ over $K$. 
\end{remark}
\begin{corollary}
	\label{C:splitting-field} The splitting field $E$ of $(M,\varphi)$ is a finite Galois extension of $K$ generated by the coefficients $u_{ij}$ of the matrix $U$ of \eqref{E:Lang}. 
\end{corollary}
\begin{proof}
	The extension $E/K$ is finite, separable by Proposition \ref{P:separability}. It is normal since a Galois conjugate of a solution $X$ of $A X^{(q)}=X$ is also a solution. Every solution $X$ of $A X^{(q)}=X$ is an $\F_q$-linear combination of the columns of $U$ by Proposition \ref{P:dim}, thus the coefficients $u_{ij}$ of $U$ generate $E$ over $K$. 
\end{proof}

\subsection{The Galois group of a Frobenius module}\label{SS:galgroup}

The Lang-Steinberg theorem (see \cite[Theorem 1]{Lang:1956ys} and \cite[Theorem 10.1]{Steinberg:1968aa}) plays an important role in the theory of Frobenius modules. 
\begin{theorem}
	[Lang-Steinberg]\label{T:lang} Let $\Gamma\subset \GL_n$ be a closed connected algebraic subgroup defined over $\F_q$ and let $A\in \Gamma(K)$. Then there exists $U\in \Gamma(\bbar{K})$ such that $U {(U^{(q)})}^{-1}=A$. 
\end{theorem}
\begin{remark}
	In fact, the element $U$ given in Theorem \ref{T:lang} lies in $\Gamma(K_\sepa)$ as discussed in Corollary \ref{C:separability}. 
\end{remark}

Next we state two theorems due to Matzat \cite{Matzat:2004kx} that play an important role in the determination of the Galois group of a Frobenius module. 
\begin{theorem}
	[``Upper Bound'' Theorem {\cite[Theorem 4.3]{Matzat:2004kx}}] \label{T:upperbound} 
	
	Let $\Gamma\subset \GL_n$ be a closed connected algebraic subgroup defined over $\F_q$ and let $A\in \Gamma(K)$. Let $E/K$ be the splitting field of the Frobenius module $(K^n,\varphi_A)$ defined by $A$ and let $U\in \Gamma(E)$ be an element given by the Lang-Steinberg theorem. Then the map
	\[ 
	\begin{aligned}
		\Gal(E/K)&\overset{\rho}{\longrightarrow} \Gamma(\F_q)\\
		\sigma &\longmapsto U^{-1}\sigma(U) 
	\end{aligned}
	\]
	is an injective group homomorphism. 
\end{theorem}

We state next Matzat's ``lower bound'' theorem in the particular case that we will use. See \cite[Theorem 3.4]{Albert:2011vn} and ensuing paragraph. 
\begin{theorem}
	[``Lower Bound'' Theorem]\label{T:lowerbound} Let $K=\F_q(\b{t})$ where $\b{t}=(t_1,\ldots,t_m)$ are indeterminates. Let $\Gamma\subset \GL_n$ be a closed connected algebraic subgroup defined over $\F_q$ and let $A\in \Gamma(K)$. Let $\rho :\Gal(E/K)\rightarrow \Gamma(\F_q)$ be the homomorphism of Theorem \ref{T:upperbound}. Then every specialization of $A$ in $\F_q$ is conjugate in $\Gamma(\bbar{\F}_q)$ to an element of $\im(\rho)$. 
\end{theorem}
\subsection{Integrality}

In this subsection we discuss integrality properties of the solutions of the system $A X^{(q)}=X$. 
\begin{proposition}
	\label{P:integrality} Let $R$ be a Noetherian domain containing $\F_q$ with field of fractions $K$ and let $A\in \GL_n(R)$. Then the solutions of the system $AX^{(q)}=X$ have coordinates that are integral over $R$. 
\end{proposition}
\begin{proof}
	Define recursively $B_0=I$, $B_k={(A^{-1})}^{(q^{k-1})} B_{k-1}$ for $k\ge 1$. Let $N_k$ the $R$-submodule of $M_n(R)$ generated by $B_0,B_1,\ldots, B_k$. Since $R$ is Noetherian, the ascending chain of submodules $\{N_k\}$ stabilizes, that is $N_{k-1}=N_k$ for $k$ large enough. For such a $k$ we have
	\[ B_k=\sum_{j=0}^{k-1} c_j B_j, \]
	where $c_j\in R$. Let $X\in K_\sepa^n$ be such that $AX^{(q)}=X$. It follows from the definition of the $B_j$'s that $X^{(q^j)}= B_j X$, thus
	\[ X^{(q^k)}= \sum_{j=0}^{k-1} c_j X^{(q^{j})}, \]
	which shows that the coordinates of $X$ are roots of the monic additive polynomial with coefficients in $R$
	\[ T^{q^k}-\sum_{j=0}^{k-1} c_j T^{q^{j}}. \]
\end{proof}
\begin{proposition}
	\label{P:integral-galois}
	
	Let $R$ be a Noetherian integrally closed domain containing $\F_q$ with field of fractions $K$ and let $A\in \GL_n(R)$. Let $U\in \GL_n(K_\sepa)$ be such that $A=U {(U^{(q)})}^{-1}$ and let $S=R[U]$ be the ring generated by the coefficients of $U$ over $R$. Then the ring extension $S/R$ is Galois with Galois group $G=\Gal(E/K)$, where $E=K[U]$. 
\end{proposition}
\begin{proof}
	Let $\rho: G\to \GL_n(\F_q)$ be the homomorphism $\rho(\sigma)=U^{-1}\sigma(U)$ of Theorem \ref{T:upperbound}. Then $\sigma(U)=U\rho(\sigma)$, so $S=R[U]$ is preserved by $G$. By Proposition \ref {P:integrality}, the ring $S$ is integral over $R$ and, since $R$ is assumed to be integrally closed, we must have $S^G=R$. It remains to show that $S/R$ is unramified at maximal ideals.
	
	Let $\mathfrak{m}\subset S$ be a maximal ideal and let $\mathfrak{m}_0=R\cap \mathfrak{m}$. Let $\ell=S/\mathfrak{m}$ and $k=R/\mathfrak{m}_0$. Notice that since $S/R$ is integral, the ideal $\mathfrak{m}_0$ is also maximal \cite[Corollary 5.8]{Atiyah:1969aa}. Clearly $\ell=k[\bbar U]$ is the splitting field of the system $\bbar{A} \b{X}^{(q)}=\b{X}$ over $k$, where $\bbar{A}$ is the class of $A$ modulo $\mathfrak{m}_0$. Hence $\ell/k$ is Galois. Let $G_\mathfrak{m}\subset G$ be the stabilizer of $\mathfrak{m}$. Each $\sigma\in G_\mathfrak{m}$ induces an automorphism $\bbar{\sigma}$ of $\ell/k$; we have a canonical homomorphism
	
	\[ 
	\begin{aligned}
		G_\mathfrak{m}&\overset{\pi}{\longrightarrow} \Gal(\ell/k)\\
		\sigma &\longmapsto \bbar{\sigma}. 
	\end{aligned}
	\]
	We need to verify that the map $\pi$ above is injective. Indeed, let $\bbar{\rho}: \Gal(\ell/k)\to\GL_n(\F_q)$ be the map given by $\bbar{\rho}(\tau)=\bbar{U}^{-1}\tau(\bbar{U})$. We verify immediately that the following diagram is commutative
	\[ 
	\begin{diagram}
		G_\mathfrak{m}&\rTo^{\rho} &\GL_n(\F_q)\\
		\dTo^{\pi} &\ruTo_{\bbar{\rho}} &\\
		\Gal(\ell/k) & &. 
	\end{diagram}
	\]
	Since $\rho$ is injective by Theorem \ref{T:upperbound}, we conclude that so is $\pi$. 
\end{proof}

\subsection{Description of the splitting field}

Let $K$ be a field containing $\F_q$ and let $A\in \GL_n(K)$.

Let $\b{U}=(\b{u_{ij}})$, where the $\b{u_{ij}}$'s ($i,j=1,\ldots,n$) are indeterminates. Let $d=\det(A)$ and let $J\subset K[\b{U}]$ be the ideal
\[ J= \langle A\b{U}^{(q)}-\b{U},\det(\b{U})^{(q-1)}d-1\rangle. \]
\begin{proposition}
	\label{P:galois-algebra} The $K$-algebra
	\[ \mc{E}=K[\b{U}]/J \]
	is a Galois $\GL_n(\F_q)$-algebra over $K$. Its indecomposable factors are isomorphic to the splitting field $E$ of the Frobenius module $(K^n,\Phi_A)$. 
\end{proposition}
\begin{proof}
	Let $U\in \GL_n(K_\sepa )$ be such that $A=U{U^{(q)}}^{-1}$ and let $\b{W}=U^{-1} \b{U}$. Then, as in Proposition \ref{P:separability}, we have
	\[ 
	\begin{aligned}
		{\mc{E}}\otimes_K K_\sepa&=K_\sepa[\b{W}]/\langle \b{W}^{(q)}-\b{W}, \det(\b{W})^{q-1}-1 \rangle\\
		&\simeq \prod_{\GL_n(\F_q)} {K_\sepa}. 
	\end{aligned}
	\]
	Thus $\mc{E}$ is {\'e}tale. The action of $\GL_n(\F_q)$ on $\mc{E}$ is given by $\b{U}\mapsto \b{U} a$ for $a\in \GL_n(\F_q)$. The primitive idempotents of ${\mc{E}}\otimes_K K_\sepa$ are represented by $e_b(\b{U})=f_b(U^{-1} \b{U})$, where $b\in \GL_n(\F_q)$ and $f_b(\b{W})\in \F_q[W]$ is the Lagrange interpolation polynomial such that $f_b(w)=\delta_{b,w}$ for $w\in \GL_n(\F_q)$. We see easily that $a e_b=e_{ba^{-1}}$, so $\GL_n(\F_q)$ acts simply transitively on the set of primitive idempotents of ${\mc{E}}\otimes_K K_\sepa$. Thus ${\mc{E}}$ is a Galois $\GL_n(\F_q)$-algebra.
	
	The indecomposable factors of $\mc{E}$ are precisely the images of $K$-algebra homomorphisms $\mc{E}\to K_\sepa$. If $\varphi: \mc{E}\to K_\sepa$ is such a homomorphism, then the columns of $U=\varphi(\b{U})$ form a $\F_q$-basis of the space of solutions of the system $A X^{(q)}=X$. Thus $E=\varphi(\mc{E})$. 
\end{proof}
Let $\{\epsilon_1,\epsilon_2,\ldots,\epsilon_h\}$ be the set of primitive idempotents of $\mc{E}$. The group $\GL_n(\F_q)$ acts on this set transitively and each subalgebra $\mc{E}\epsilon_i$ (with identity $\epsilon_i)$ is isomorphic to $E$ by Proposition \ref{P:galois-algebra}. 
\begin{proposition}
	Let $R$ be an integrally closed Noetherian domain with field of fractions $K$ and let $\mc{S}=R[\b{U}]/J_0$, where $J_0=J\cap R[\b{U}]$. Assume $A\in \GL_n(R)$. Then each primitive idempotent $\epsilon_i$ of $\mc{E}$ lies in $\mc S$. In particular, we have a decomposition 
	\begin{equation}
		\label{E:decomposition} \mc{S}=\sum_{i=1}^h \mc{S}\epsilon_i. 
	\end{equation}
\end{proposition}
\begin{proof}
	It is enough to prove that $\epsilon_1\in \mc{S}$. Let $G$ be the stabilizer of $\epsilon_1$ in $\GL_n(\F_q)$. Then
	\[ \epsilon_1=\sum_{a\in G} e_a, \]
	where the $e_a\in \mc{E}\otimes K_\sepa$ are absolutely primitive idempotents. As we have seen in the proof of Proposition \ref{S:generic-pols}, we have $e_a(\b{U})=f_b(U^{-1}\b{U})$, where $f_a(\b{W})\in\F_q[\b{W}]$ is the Lagrange interpolation polynomial such that $f_a(w)=\delta_{a,w}$ for $w\in \GL_n(\F_q)$, where $\delta$ is the Dirichlet symbol. Since the entries of $U$ (and $U^{-1}$) are integral over $R$ by Proposition \ref{P:integrality}, we conclude that the coefficients of $e_a(\b{U})$, as polynomial in the variables $\b{u}_{ij}$, are integral over $R$. It follows that $\epsilon_1\in K[\b{U}]$ has coefficients integral over $R$. Since $R$ is integrally closed by hypothesis, we have $\epsilon_1\in R[\b{U}]$. 
\end{proof}
\begin{corollary}
	The ring extension $\mc{S}/R$ is Galois with group $\GL_n(\F_q)$. 
\end{corollary}
\begin{proof}
	Let $\epsilon\in\mc{S}$ be a primitive idempotent and let $G$ be the stabilizer of $\epsilon$ in $\GL_n(\F_q)$ and let $S=\mc{S}\epsilon$. From the decomposition \eqref{E:decomposition} we have
	\[ \mc{S}\simeq\Map_G(\GL_n(\F_q), S), \]
	where $\Map_G(\GL_n(\F_q), S)$ is the set of $G$-equivariant maps $\GL_n(\F_q)\to S$ and $(a\alpha)(x)=\alpha(xa)$ for $a\in \GL_n(\F_q)$ and $\alpha\in\Map_G(\GL_n(\F_q)$. Since $S/R$ is $G$-Galois by Proposition \ref{P:integral-galois}, we conclude that $\mc{S}/R$ is $\GL_n(\F_q)$-Galois. 
\end{proof}

\section{Generic extensions for multiplicative groups}\label{S:generic-pols}

Let $k$ be a field and let $G$ be a finite group. Let let $R=k[\b{t},1/d]$, where $\b{t}=(t_1,\ldots,t_m)$ are indeterminates and $d$ is a nonzero polynomial in $k[\b{t}]$. The following definition is due to Saltman \cite{Saltman:1982uq}. 
\begin{definition}
	A Galois $G$-extension $S/R$ of commutative rings is called {\em $G$-generic over $k$} if for every Galois $G$-algebra $M/L$, where $L$ is a field containing $k$, there exists a homomorphism of $k$-algebras $\varphi: R\to L$ such that $S\otimes_\varphi L \simeq M$ as $G$-algebras over $L$. 
\end{definition}

In this section $\A\subset M_n(\F_q)$ denotes a fixed $\F_q$-subalgebra and $m$ denotes its dimension over $\F_q$. The goal of this section is to construct explicitly a Galois $\A^\times$-extension $S/R$ that is $\A^\times$-generic in the above sense. 

We denote henceforth by $\G $ the multiplicative group $\G _m(\A)$ as an algebraic group defined over $\F_q$. Let $a_1,a_2,\ldots,a_m$ be a basis of $\A$ over $\F_q$ and define 
\begin{equation}
	\label{E:main-matrix} A(\b{t})=\sum_{i=1}^m t_i a_i, 
\end{equation}
where $\b{t}=(t_1,\ldots,t_m)$ are indeterminates.

Let $d=\det(A)$ and let $R=\F_q[\b{t},1/d]$. By the construction of $R$ we clearly have $A\in \G(R)$. Let $E$ be the splitting field of the Frobenius module given by $A$ over $K=\F_q(\b{t})$. By Theorem \ref{T:lang}, there exists $U\in \G (K_\sepa)$ such that $A=U (U^{(q)})^{-1}$. Recall that by Corollary \ref{C:splitting-field}, the coefficients $u_{ij}$ of $U$ generate $E$ over $K$. We write, by abuse of notation, $E=K(U)$. We define similarly $S=R[U]$, the subring of $E$ generated by the $u_{ij}$'s over $R$. Note that by Proposition \ref{P:integrality} the $u_{ij}$'s are integral over $R$, so $S$ is finitely generated as an $R$-module.

Here is the main theorem in this section. 
\begin{theorem}
	\label{T:main} With the notation above, we have 
	\begin{enumerate}
		\item $\Gal(E/K)\simeq\G (\F_q).$ 
		\item The ring extension $S/R$ is $\G (\F_q)$-generic. 
	\end{enumerate}
\end{theorem}
The following two lemmas will be needed in the proof of Theorem \ref{T:main}. 
\begin{lemma}
	\label{L:conj} Let $a,b\in \G (\F_q)$. If $a$ and $b$ are conjugate in $\G (\bbar{\F}_q)$, then they are conjugate in $\G (\F_q)$. 
\end{lemma}
\begin{proof}
	Suppose $a=u b u^{-1}$ with $u\in \G (\bbar{\F}_q)$. Let $\sigma\in \Gal(\bbar{\F}_q/\F_q)$. Then $z_\sigma:=u^{-1}\sigma(u)$ is in ${(\mc{Z}\otimes \bbar{\F}_q)}^\times$, where $\mc{Z}$ is the centralizer of $b$ in $\A$. The map $\sigma\mapsto z_\sigma$ is a $1$-cocycle with values in ${(\mc{Z}\otimes \bbar{\F}_q)}^\times$. By the generalized Hilbert Theorem 90 (see e.g. \cite[Chap. X]{Serre:1968zr}), this $1$-cocycle is trivial, that is, there exists $w\in {(\mc{Z}\otimes \bbar{\F}_q)}^\times$ such that $z_\sigma:=w^{-1}\sigma(w)$ for all $\sigma\in \Gal(\bbar{\F}_q/\F_q)$. Then $v:=u w^{-1}$ satisfies $a=v b v^{-1}$ and is fixed under $ \Gal(\bbar{\F}_q/\F_q)$, that is, $v$ is in $\G (\F_q)=\A^\times$. 
\end{proof}
\begin{lemma}
	\label{L:Jordan} Let $G$ be a finite group and let $C_1,C_2,\ldots,C_h$ be the conjugacy classes of $G$. Let $g_i\in C_i$ for $i=1,\ldots,h$. Then the set $\{g_1,g_2,\ldots,g_h\}$ generates $G$. 
\end{lemma}
\begin{proof}
	See \cite[Theorem 4']{Serre:2003ly}. 
\end{proof}
\begin{proof}
	[Proof of Theorem \ref{T:main}] (1) By Theorem \ref{T:lang}, there exists $U\in \G(K_{\sepa})$ such that $A=U {U^{(q)}}^{-1}$. Let $\rho: \Gal(E/K)\to \G (\F_q)$ be the map defined by $\rho(\sigma)= U^{-1} \sigma(U)$. By Theorem \ref{T:upperbound}, the map $\rho$ is an injective group homomorphism.
	
	We have on the one hand by Theorem \ref{T:lowerbound} and Lemma \ref{L:conj} that every specialization $A(\bb\xi)\in \G (\F_q)$ (where $\bb\xi\in \F_q^m$) is conjugate in $\G (\F_q)$ to an element of $\mathrm{im}(\rho)$. On the other hand, every element of $\G (\F_q)$ is of the form $A(\bb\xi)$ for some $\bb\xi\in \F_q^m$, thus every conjugacy class of $\G (\F_q)$ intersects nontrivially $\mathrm{im}(\rho)$. We conclude by Lemma \ref{L:Jordan} that $\mathrm{im}(\rho)=\G (\F_q)$.
	
	(2) Let $L$ be a field containing $\F_q$ and let $M/L$ be a Galois $G$-algebra with group $G=\G (\F_q)$ and let $\delta\in M$ be a primitive idempotent. Then $N=M\delta$ is a field that is Galois with group $H=G_\delta$ over $K$. Moreover, there is an isomorphism of $G$-algebras over $L$
	\[ M\simeq \Map_H(G,N), \]
	where $\Map_H(G,N)$ is the algebra of $H$-equivariant maps $G\to N$ \cite[Proposition 18.18]{knus:1998fk}. The action of $G$ is given by $(g\alpha)(x)=\alpha(x g)$ for $\alpha\in \Map_H(G,N)$ and $g\in G$. Under the above isomorphism, the primitive idempotents of $M$ correspond to the characteristic functions of the right cosets of $H$ in $G$. In particular, $\delta$ corresponds to the characteristic function of $H$.
	
	Let $\rho: \Gal(N/L)\to H$ be an isomorphism. Composing with the inclusion $H\subset G=\G(\F_q)\subset \G (N)$, we can view $\rho$ as a $1$-cocycle with values in $\G (N)=(\A\otimes N)^\times$. By the generalized Hilbert Theorem 90 (see e.g. \cite[Chap. X]{Serre:1968zr}), $\rho$ is a trivial $1$-cocycle, i.e. there exits $W\in \G (N)$ such that $\rho(\sigma)=W^{-1}\sigma(W)$ for all $\sigma \in \Gal(N/L)$ .
	
	We first observe that $N$ is generated over $L$ by the coefficients $w_{ij}$ of $W$. Indeed, if $\sigma\in G$ is such that $\sigma(w_{ij})=w_{ij}$ for $i,j=1,\ldots,n$, then $\rho(\sigma)=1$ and consequently $\sigma=1$ since $\rho$ is injective.
	
	Let $B=W {W^{(q)}}^{-1}$. It is readily verified that $B$ is fixed by $\Gal(N/L)$ and hence lies in $\G (L)$. Thus we can write $B=A(\bb\xi)$ for some $\bb\xi=(\xi_1,\ldots,\xi_n)\in L^n$. Define a $\F_q$-algebra homomorphism $f: R\to L$ by $\b{t}\mapsto \bb\xi$. Since $S$ is integral over $R$, we can extend $f$ to a ring homomorphism \cite[Ch. VII, Proposition 3.1]{Lang:2002ve}
	\[ \hat{f}: S \rightarrow \overline{L}. \]
	Let $U$ be the class of $\b{U}$ in $S$. Then $U ({U^{(q)}})^{-1}=A$ and $W_1:=\hat{f}(U)\in \GL_n(\overline{L})$ satisfies $W_1 ({W_1^{(q)}})^{-1}=B$ so $W_1=Wg$ for some $g\in \GL_n (\F_q)$. Replacing $U$ by $U g^{-1}$, we can assume $ \hat{f}(U)=W$. Since the coefficients $u_{ij}$ of $U$ generate $S$ over $R$ and the coefficients $w_{ij}$ of $W$ generate $N$ over $L$, we have 
	\begin{equation}
		\label{E:surj} \hat{f}(S)L=N. 
	\end{equation}
	Since $\hat{f}$ is $\F_q$-linear, we have
	\[ \hat{f}(U h)= Wh \]
	for $h\in H$. Identifying $\Gal(S/R)$ with $G$ via the isomorphism $\sigma\mapsto U^{-1}\sigma(U)$ and $\Gal(N/L)$ with $H$ via the isomorphism $\tau\mapsto W^{-1}\tau(W)$, we have from above that
	\[ \hat{f}(h(U))= h(W) \]
	for $h\in H$, which implies that $\hat{f}$ is an $H$-homomorphism. Then we can consider the induced $G$-homomorphism
	\[ F: S\longrightarrow M=\Map_H(G,N) \]
	defined by $F(s)(g)=\hat{f}(g(s))$ for $s\in S$ and $g\in G$. 
	
	Since $S/R$ is $G$-Galois by Proposition \ref{P:integral-galois}, so is $S\otimes_f L/L$ and the map 
	\begin{equation}
		\label{E:iso} 
		\begin{aligned}
			S\otimes_f L & \longrightarrow M\\
			s\otimes x &\longmapsto F(s)x 
		\end{aligned}
	\end{equation}
	is a morphism of Galois $G$-extensions of $L$, which is automatically an isomorphism (see e.g. \cite[Proposition 5.1.1]{Jensen:2002bh}). 
\end{proof}

\section{Generic Polynomials}

We recall here the definition of generic polynomial. We refer to \cite{Jensen:2002bh} for details and a wealth of examples.

Let $\b{t} = (t_1,\dots t_m)$ be indeterminates over the field $k$ and let $G$ be a finite group. 
\begin{definition}
	A monic separable polynomial $f(Y;\b{t}) \in k(\b{t})[Y]$ is called {\em $G$-generic } over $k$ if the following conditions are satisfied: 
	\begin{enumerate}
		\item $\Gal(f(Y;\b{t})/k(\b{t}))\simeq G$. 
		\item Every Galois $G$-extension $M/L$, where $L$ is a field containing $k$, is the splitting field of a specialization $f(Y;\bb\xi)$ for some $\bb\xi \in L^n$. 
	\end{enumerate}
\end{definition}

In this section we give a method to explicitly construct a generic polynomial for the group $G=\A^\times$ over the field $k=\F_q$. The method is based on the cyclicity of Frobenius modules over $k(\b{t})$ (see \cite[Section I.2] {Matzat:2004kx}). 
\begin{definition}
	A Frobenius module $(M, \varphi)$ over a field $K$ is {\em cyclic} if there exists a nonzero vector $v\in M$ such that $\{v, \varphi(v), \varphi ^2(v), \ldots, \varphi ^{n-1}(v)\}$ forms a basis of $M$. 
\end{definition}
Note that the matrix of $(M,\varphi)$ relative to a cyclic basis
\[\{v, \varphi(v), \varphi ^2(v), \ldots, \varphi ^{n-1}(v)\}\]
has the form 
\begin{equation}
	\label{delta} \Delta= 
	\begin{pmatrix}
		0 & 0 & \cdots &0 & a_0 \\
		1 & 0 &\cdots&0 &a_1\\
		0 & 1 &\ddots&\vdots &\vdots\\
		\vdots & \ddots&\ddots &0 &\vdots\\
		0 & \cdots&0 & 1&a_{n-1}\\
	\end{pmatrix}
	. 
\end{equation}

In \cite[Theorem 2.1]{Matzat:2004kx}, Matzat proves in particular that if the ground field $K$ is infinite, all Frobenius modules over $K$ are cyclic. The Frobenius modules we consider in this section are over the field $K=\F_q(\b{t})$, where $\b{t}=(t_1,\ldots,t_m)$, so they are always cyclic.

For $B\in \GL_n(K)$, we denote by $B^*$ the matrix
\[ B^* = (B^{-1})^T. \]
Notice that the map $B\mapsto B^*$ is a group homomorphism.
\begin{proposition}
	\label{gsf}Let $B\in \GL_n(K)$. The systems $BX^{(q)} =X$ and $B^*X^{(q)} =X$ have the same splitting fields. 
\end{proposition}
\begin{proof}
	Let $U\in \GL_n(K_{\sepa})$ be such that $B=U(U^{(q)})^{-1}$. As we have seen in \ref{C:splitting-field}, the splitting field of the Frobenius module given by $B$ is found by adjoining the coefficients of $U$ to the base field $K$. If we apply the matrix operator $^*$, we obtain $$B^*=U^*({U^*}^{(q)})^{-1},$$ which shows that the splitting field of the Frobenius module given by $B^*$ is generated over $K$ by the coefficients of $U^*$. Clearly the coefficients of $U$ and those of $U^*$ generate the same field. 
\end{proof}

Let $B\in\GL_n(K)$, where $K$ is an infinite field. Then the Frobenius module $(K^n,\varphi_B)$, where $\varphi_BX=BX^{(q)}$ admits a cyclic basis, that is, there exists $N\in \GL_n(K)$ such that 
\begin{equation}
	\label{E:change-basis} N^{-1} B N^{(q)}=\Delta, 
\end{equation}
Where $\Delta$ is a matrix of the form \eqref{delta}. An immediate application of Proposition \ref{gsf} is 
\begin{corollary}
	\label{sf} The splitting fields of the Frobenius modules given by $B$ and $\Delta^*$ are the same. 
\end{corollary}
\noindent 

Computing the splitting field of the Frobenius module given by $\Delta^*$ is straightforward. We solve explicitly the system $\Delta^* X^{(q)}=X$ or, equivalently, the system $X^{(q)}=\Delta^T X$. Letting $X=(x_1,\ldots,x_n)^T$, we have
\[ 
\begin{cases}
	x_1 ^q &= x_2 \\
	&\vdots\\
	x_{n-1} ^q &= x_{n}\\
	x_{n} ^q &= a_0x_1+a_1x_2+\cdots+a_{n-1}x_{n}. 
\end{cases}
\]
Setting $x_1=y$, we have from the above system $x_i=y^{q^{i-1}}$ for $i=1,\ldots,n$, where $y$ satisfies the equation
\[ y^{q^n}=a_0y+a_1y^q+\cdots+a_{n-1}y^{q^{n-1}}. \]
\begin{corollary}
	\label{spA} The splitting fields of the Frobenius module given by $B$ and the additive polynomial $f(Y)=Y^{q^n}-a_0y-a_1Y^q-\cdots-a_{n-1}Y^{q^{n-1}}$. coincide. 
\end{corollary}
\begin{remark}
	The polynomial $f(Y)$ above is separable since $f'(Y)=a_0=\det\Delta\ne 0$. 
\end{remark}

We shall now apply the above observations to obtain an explicit generic polynomial for the group $\A^\times$, where $\A\subset M_n(\F_q)$ is a $\F_q$-subalgebra. Recall that $\G $ denotes the multiplicative group $\G _m(\A)$ as an algebraic group defined over $\F_q$. Let $v_1,v_2,\ldots,v_m$ be a basis of $\A$ over $\F_q$ and define
\[ A(\b{t})=\sum_{i=1}^m t_i v_i, \]
where the $t_i$'s are indeterminates.

Our next goal is to show that for $K=\F_q(\b{t})$ and $B=A(\b{t})$, the polynomial $f\in K[Y]$ given by Corollary \ref{spA} is $\G (\F_q)$-generic. We will need the following preliminary lemmas. 
\begin{lemma}
	\label{L:lang2} Let $L$ be a field and let $B\in \G (L)$. Then the morphism of affine varieties defined over $L$ 
	\begin{equation}
		\begin{diagram}
			\psi: & \G &\rTo & \G \\
			& X &\rMapsto &X^{-1} B X^{(q)}\\
		\end{diagram}
	\end{equation}
	is an epimorphism, that is, the induced ring homomorphism $\psi^*: L[\G ]\to L[\G ]$ is injective. 
\end{lemma}
\begin{proof}
	Over an algebraic closure $\bar{L}$, the map $\psi: \G (\bar{L})\to \G (\bar{L})$ is surjective as an immediate consequence of the Lang-Steinberg theorem. Indeed, write $B=U {U^{(q)}}^{-1}$ with $U\in \G (\bar{L})$ and let $Y=U^{-1}X$. Then $\psi(X)=Y^{-1} Y^{(q)}$. Theorem \ref{T:lang} states that all elements of $\G (\bar{L})$ are of the form $Y^{-1} Y^{(q)}$.
	
	Thus the induced ring homomorphism $\psi^*: \bar{L}[\G ]\to \bar{L}[\G ]$ is injective. The announced result follows trivially from this. 
\end{proof}
\begin{lemma}
	\label{L:especial} Assume that $L$ is an infinite field. Let $p\in L[\b{t},1/d]$ be a nonzero rational function and let $B$ be an element of $\G (L)$. Then there exists $\bb\xi\in L^n$ such that $p(\bb\xi)\ne 0$ and $A(\bb\xi)$ is Frobenius-equivalent to $B$ in $\G (L)$. 
\end{lemma}
\begin{proof}
	Let $O\subset \mathbb{A}^m$ be the open subset where $d\ne 0$. Then the the map $\alpha: O\to \G $ given by $\alpha(\bb\xi)=A(\bb\xi)$ is an isomorphism of affine varieties defined over $L$. Define $\varphi=\alpha^{-1}\circ\psi\circ \alpha$. By Lemma \ref{L:lang2}, $ \varphi^*(p)=p\circ \varphi$ is not zero. Since $L$ is infinite, there exists $\bb\eta\in O(L)\subset L^m$ such that $p(\varphi(\bb\eta))\ne 0$. Let $\bb\xi=\varphi(\bb\eta)$. Then $A(\bb\xi)=\alpha(\bb\xi)=\alpha(\varphi(\bb\eta))=\psi(\alpha(\bb\eta))=\alpha(\bb\eta) ^{-1} B \alpha(\bb\eta)^{(q)}$. 
\end{proof}
\begin{theorem}
	\label{T:main-pol} Let $f(Y;\b{t})\in \F_q(\b{t})[Y]$ be the polynomial obtained from $A(\b{t})$ as in Corollary \ref{spA}. Then $f(Y;\b{t})$ is $\G(\F_q)$-generic over any infinite field $k$ containing $\F_q$. 
	\begin{proof}
		Let $K=\F_q(\b{t})$ and let $E/K$ be the splitting field of the Frobenius module $(K^n, \varphi_{A\b{t})})$. By Corollary \ref{spA}, $E$ is also the splitting field of $f(Y;\b{t})$. We already know by Theorem \ref{T:main} that $\Gal(E/K)\simeq \G (\F_q)$. Thus we need only to show that $f(Y;\b{t})$ is generic.
		
		As in \eqref{E:change-basis}, there exists $N\in \GL_n(K)$ such that 
		\begin{equation}
			\label{E:equiv} N^{-1}A N^{(q)} = \Delta. 
		\end{equation}
		By choosing a cyclic basis $b\in R^n$ (where $R=\F_q[\b{t},1/d]$ as in Section \ref{S:generic-pols}), we can assume that $N$ has coefficients in $R$. Let $p(\b{t})=\det N$.
		
		Let $M/L$ be a $\G(\F_q)$-extension, where $L$ is an infinite field containing $\F_q$. Choose an isomorphism $\rho: \Gal(M/L)\overset{\simeq}{\to} \G(\F_q)$. We view $\rho$ as a $1$-cocycle with values in $\G(M)$. By the general Hilbert's Theorem 90 \cite[Chap. X]{Serre:1968zr}, there exists $W\in \G(M)$ such that $\rho(\sigma)=W^{-1}\sigma(W)$ for $\sigma\in \Gal(M/L)$. Define $B=W {W^{(q)}}^{-1}$. An elementary verification shows that $B$ is fixed under $\Gal(M/L)$ and therefore lies in $\G(L)$. It is also easy to see that $M$ is the splitting field of the system $B X^{(q)}=X$. By Lemma \ref{L:especial}, there exists $\bb\xi\in L^n$ such that $p(\bb\xi)\ne 0$ and $B':=A(\bb\xi)$ is Frobenius-equivalent to $B$. Since $\bb\xi$ has been chosen so that $N(\bb\xi)$ is nonsingular (recall that $p(\b{t})=\det N$), we can evaluate \eqref{E:equiv} at $\b{t}=\bb\xi$. We get 
		\begin{equation}
			N(\bb\xi)^{-1}B' N(\bb\xi)^{(q)} = \Delta(\bb\xi). 
		\end{equation}
		We conclude by Corollary \ref{spA} that $M$ is the splitting field of $f(Y;\bb\xi)$ over $L$. 
	\end{proof}
\end{theorem}

\section{Examples} In this section, we give specific examples of generic polynomials. 
\begin{example}
	Let $\A =\F_9$ be seen as finite-dimensional algebra over $\F_3$. Then $G=\A^\times \simeq C_8$.\\
	Taking the basis $\{1,\sqrt{-1}\}$ of $\A$ over $\F_3$, we can embed $\A$ into $M_2(\F_3)$ via the regular representation. Then the matrix $A$ of \eqref{E:main-matrix} is given by
	\[ A(\b{t})= 
	\begin{pmatrix}
		t_1 & -t_2 \\t_2 & t_1\ 
	\end{pmatrix}
	. \]
	Let $v=(0,1)^T \in \F_3 ^2$ serve as the generator for the cyclic module. Then as in the last section,
	\[ N = (v|Av^{(3)}) = 
	\begin{pmatrix}
		1 & t_1\\
		0&t_2 
	\end{pmatrix}
	. \]
	Clearly $N$ is non-singular. Let
	\[ \Delta =N^{-1}AN^{(3)} = 
	\begin{pmatrix}
		0 & - t_2^2 \left(t_1^2+t_2^2\right) \\
		1 & t_1 \left(t_1^2+t_2^2\right) 
	\end{pmatrix}
	. \]
	By Theorem \ref{T:main-pol}, the additive polynomial $f$ below build with the coefficients of the last column of $\Delta$ is generic for the group $C_8$ over any infinite field of characteristic $3$. 
	
	\[ f(Y;\b{t}) = t_2^2 \left(t_1^2+t_2^2\right)Y-t_1\left(t_1^2+t_2^2\right)Y^3+Y^9. \]
	
	This computation generalizes easily for any odd prime $p$. An additive generic polynomial for $C_{p^2-1}$ in characteristic $p$ is
	
	\[ f(Y;\b{t}) = t_2^{p-1}\left(t_1^2-\varepsilon t_2^2\right)Y-t_1\left(t_1^{p-1}+t_2^{p-1}\right)Y^p+Y^{p^2}, \]
	
	where $\varepsilon\in \F_p^\times$ is a nonsquare. 
\end{example}
\begin{example}
	Consider the following matrices $\GL_3(\F_2)$:
	\[a = 
	\begin{pmatrix}
		1&1&0 \\
		0&1&0\\
		0&0&1 
	\end{pmatrix}
	,\quad b = 
	\begin{pmatrix}
		1&0&1 \\
		0&1&0\\
		0&0&1 
	\end{pmatrix}
	,\quad c = 
	\begin{pmatrix}
		1&1&0 \\
		0&1&1\\
		0&1&0 
	\end{pmatrix}
	. \]
	It is easily verified that they generate a subgroup isomorphic to $A_4$, the alternating group on four elements.
	
	Let $\A$ be the subalgebra generated by $a$, $b$ and $c$ in $M_3(\F_2)$. We verify readily that $\dim_{\F_2} (\A) =5$ and $|\A^\times| = 12$. Thus $\A^\times\simeq A_4$. After choosing a basis of $\A$, we obtain a matrix in $5$ parameters
	\[ A(\b{t})= 
	\begin{pmatrix}
		t_1+t_2+t_3+t_4+t_5 & t_2 &t_3+t_4 \\
		0 & t_1+t_2+t_4+t_5 &t_2+t_3+t_5\\
		0 & t_2+t_3+t_5 &t_1+t_3+t_4\\
	\end{pmatrix}
	. \]
	As in the last section, we choose a generator for the associated Frobenius module. Let $v = (1,0,1)^T \in \F_2 ^3$. The matrix
	\[ N = (v|Av^{(2)}|AA^{(2)}v^{(4)}). \]
	is nonsingular, so $v$ is indeed a generator.
	
	As before, we compute $\Delta = N^{-1}AN^{(2)}$. Recall that the entries in the last column of $\Delta$ are the coefficients of an additive generic polynomial $f(Y;\b{t})$ of degree $8$ for $\A^\times\simeq A_4$ by Theorem \ref{T:main-pol}. We exhibit below an irreducible factor $g$ of $f(Y;\b{t})$ of degree $4$. Since no proper quotient of $A_4$ can act transitively on $4$ elements, the Galois group of $g$ over $\F_2(\b{t})$ is $A_4$. Obviously $g$ is also generic.
	\[ 
	\begin{aligned}
		g=&Y^4+(t_1^2+t_1t_2+t_2^2+t_1t_3+t_2t_3+t_3^2+t_2t_4+t_3t_4+t_4^2+t_1t_5+t_3t_5+\\
		&t_4t_5+t_5^2)Y^2+(t_1^2t_2+t_1t_2^2+t_2^3+t_1^2t_3+t_1t_3^2+t_3^3+t_2^2t_4+t_3^2t_4+\\
		&t_2t_4^2+t_3t_4^2+t_1^2t_5+t_2^2t_5+t_4^2t_5+t_1t_5^2+t_2t_5^2+t_4t_5^2+t_5^3)Y+\\
		&(t_1^2t_2t_4+t_2^3t_4+t_1^2t_3t_4+t_1t_2t_3t_4+t_1t_3^2t_4+t_2t_3^2t_4+t_1^2t_4^2+\\
		&t_2^2t_4^2+t_2t_3t_4^2+t_2t_4^3+t_3t_4^3+t_4^4+t_1t_2^2t_5+t_2^3t_5+t_2^2t_3t_5+\\
		&t_1t_3^2t_5+t_2t_3^2t_5+t_3^3t_5+t_1^2t_4t_5+t_1t_3t_4t_5+t_3t_4^2t_5+t_4^3t_5+\\
		&t_2^2t_5^2+t_3^2t_5^2+t_2t_4t_5^2+t_4^2t_5^2+t_1t_5^3+t_2t_5^3+t_3t_5^3+t_5^4). 
	\end{aligned}
	\]
	While this method always produces $\A^*$-generic polynomials, the number of parameters is not optimal. A generic polynomial with two parameters was obtained in \cite{Sergeev:2006fk} for $A_4$, compared to the five parameters that this method needed. 
	
	The function field in one variable $\F_2(s)$ is Hilbertian, so ``most'' specializations of $g$ in $\F_2(s)$ are irreducible and have Galois group $A_4$. Here are some examples.
	\[ 
	\begin{aligned}
		g_1&=s + Y + Y^2 + Y^4;\\
		g_2&=s^2 + s^3Y + s^2Y^2 + Y^4.\\
	\end{aligned}
	\]
\end{example}

\bibliographystyle{plain}

\end{document}